\documentclass[11pt]{article}
\usepackage[utf8]{inputenc}
\usepackage[T1]{fontenc}
\usepackage{amsmath,amssymb,amsthm}
\usepackage{booktabs}
\usepackage{siunitx}
\usepackage[margin=2.7cm]{geometry}
\usepackage{hyperref}

\newtheorem{theorem}{Theorem}
\newtheorem{proposition}[theorem]{Proposition}
\newtheorem{corollary}[theorem]{Corollary}
\newtheorem{remark}[theorem]{Remark}
\theoremstyle{definition}

\title{All polyhedral graphs with at most 16 vertices are Schurian:\\
a computational verification}
\author{David Reh\\
  \small Independent Researcher, Fröndenberg/Ruhr, Germany\\
  \small \texttt{DavidReh@web.de}}
\date{September 2026}

\begin{document}
\maketitle

\begin{abstract}
Li, Ponomarenko and Zeman state their belief that every polyhedral graph
(i.e.\ every $3$-connected planar graph) is \emph{Schurian}: that its
coherent closure $\mathrm{WL}(X)$ coincides with the orbital configuration
of $\mathrm{Aut}(X)$. Combined with their Theorem~1.1, this would imply
that every polyhedral graph has Weisfeiler--Leman dimension at most~$2$.
We report a computational verification of the Schurian property for all
polyhedral graphs with at most $16$ vertices, a total of
\num{413024847068} graphs, none of which is non-Schurian. The computation
combines the generator \texttt{plantri} with a one-sided $1$-WL
discreteness filter and, for every graph surviving that filter, a full
comparison of the stable $2$-WL partition of $\Omega^2$ with the orbit
partition of $\mathrm{Aut}(X)$. The computation is guarded by the published anchor counts, a seeded
re-enumeration of the two largest stages, independent \texttt{nauty}
cross-checks and control graphs with mathematically fixed verdicts;
code, checkpoints and a verification script accompany the note as
ancillary files. For $n\le 13$ the result reproduces what already follows from the
published enumeration of coherent configurations, by an entirely different
route; from $n=14$ on, deriving it from that side would require a step we
could not find in the literature, and at $n=16$ the all-graphs statement
fails outright --- the Shrikhande graph is non-Schurian --- so there the
polyhedral restriction carries the result. The statement is established by
exhaustive computation; with respect to the general conjecture it is
instance evidence, not a proof. A companion note treats the torus side,
where non-Schurian triangulation graphs do occur at $n = 16$.
\end{abstract}


\section{Introduction}\label{sec:intro}

Let $X$ be a finite simple graph with vertex set $\Omega$, $|\Omega|=n$.
The \emph{coherent closure} $\mathrm{WL}(X)$ is the coarsest coherent
configuration on $\Omega$ refining the partition of $\Omega^2$ into the
diagonal, the edges and the non-edges; it is the stable partition of the
classical $2$-dimensional Weisfeiler--Leman refinement \cite{WL1968},
\cite[\S2.6.1]{ChenPonomarenko}. The group $\mathrm{Aut}(X)$
acts componentwise on $\Omega^2$, and its orbits form a coherent
configuration as well \cite[\S2.2.2]{ChenPonomarenko}, the \emph{orbital
configuration} of $\mathrm{Aut}(X)$. Since $2$-WL colours are invariant under
$\mathrm{Aut}(X)$, the orbital configuration always refines
$\mathrm{WL}(X)$. The graph $X$ is called \emph{Schurian} if the two
coincide, i.e.\ if the coherent configuration $\mathrm{WL}(X)$ is
Schurian in the sense of \cite[Def.~2.2.9 and
Thm.~2.2.10]{ChenPonomarenko}; this is also the definition of
\cite[p.~2]{LPZ}, both sources spelling the word `schurian'
(Remark~\ref{rem:def}). A graph is \emph{polyhedral} if it is planar and
$3$-connected; by Steinitz's theorem \cite[Thm.~4.1]{Ziegler} these are
exactly the graphs of convex $3$-polytopes.

The \emph{Weisfeiler--Leman dimension} $\dim_{\mathrm{WL}}(X)$ is the least
$k$ such that $k$-dimensional WL refinement distinguishes $X$ from every
non-isomorphic graph (\cite[Def.~18.4.2]{Grohe2017} for the class
$\{X\}$). Li, Ponomarenko and Zeman
\cite{LPZ} prove that Schurity is sufficient for a low Weisfeiler--Leman
dimension in this class:

\begin{theorem}[{\cite[Thm.~1.1]{LPZ}}]\label{thm:lpz}
Let $X$ be a polyhedral graph. If $X$ is Schurian, then
$\dim_{\mathrm{WL}}(X)\le 2$.
\end{theorem}

They add, on the same page and without assigning a number to the
statement, the expectation that the hypothesis is always satisfied:
``\emph{We believe that every polyhedral graph is schurian, and if this is
true, then by Theorem 1.1 below, the WL-dimension of any polyhedral graph
is at most 2}'' \cite[p.~2]{LPZ}. This is the statement the present note
tests. Its interest lies in the gap to the best unconditional bound: the
Weisfeiler--Leman dimension of planar graphs is at most~$3$ \cite{KPS},
and lowering the bound to~$2$ on the polyhedral subclass would follow from
Schurity via Theorem~\ref{thm:lpz}. Whether the bound $3$ is attained is
open: ``the question of whether the maximum dimension of the class of
planar graphs is 2 or 3 remains open'' \cite[p.~44:3]{KPS}, and Li,
Ponomarenko and Zeman name the search for a dimension-$3$ example
among polyhedral graphs as an initial motivation of their paper
\cite[\S1]{LPZ}.

\paragraph{What is known on the configuration side.}
Schurity of \emph{all} graphs on few vertices is a corollary of the
enumeration of coherent configurations. Klin and Ziv-Av \cite[Cor.~4.1]{KlinZivAv2017}
established that all coherent configurations of order at most~$13$ are
Schurian, that exactly one non-Schurian configuration of order~$14$ exists
(rank~$11$, fibres of sizes $6$ and $8$, automorphism group of order~$24$
isomorphic to $\mathrm{SL}(2,3)$). Ziv-Av \cite{ZivAv2018} gave the detailed
report of that same enumeration; there are exactly two non-Schurian
configurations of order~$15$ \cite[Cor.~4.1]{KlinZivAv2017},
\cite[\S2.1]{ZivAv2018}. Consequently every graph on at most $13$ vertices is Schurian,
which is the bound Li, Ponomarenko and Zeman also quote \cite[p.~2]{LPZ}.

Two remarks on how far this reaches. First, for orders $14$ and $15$ the
step from ``no non-Schurian configuration is the closure of a graph'' to
``every graph is Schurian'' is not made in
\cite{KlinZivAv2017,ZivAv2018} and we are not aware of a published
argument. For one of the three the point is immediate: the order-$15$
doubly regular tournament has rank~$3$ with a single antisymmetric pair of
off-diagonal basis relations, and the edge set of an undirected graph would
have to be a symmetric irreflexive union of basis relations, forcing a
closure of rank~$2$. The substance therefore lies with the order-$14$
configuration and its order-$15$ extension by a fibre of size~$1$
\cite[\S2.1]{ZivAv2018}. In the companion note \cite{Companion} this
point is settled computationally: neither configuration is the coherent
closure of a graph (all $64$ resp.\ $256$ candidate edge unions checked),
so every undirected graph on at most $15$ vertices is Schurian. The
present note does not rely on that step --- its screening covers the
polyhedral graphs of those orders directly. Second, for order~$16$ no enumeration of coherent configurations in
the general, possibly inhomogeneous case is available; the enumeration of
\cite{ZivAv2018} stops at order~$15$ and states that larger orders require
further optimisation \cite[\S4.1]{ZivAv2018}. Association schemes --- the
homogeneous case --- are classified for orders $16$ and $17$ by Hanaki and
Miyamoto \cite{HanakiMiyamoto}, and of the $222$ association schemes of
order~$16$ exactly $16$ are non-Schurian \cite[\S3]{KlinZivAv2017}. At
order~$16$, however, the situation changes qualitatively: one of those
$16$ schemes \emph{is} the coherent closure of a graph --- it is the
rank-$3$ scheme of the Shrikhande graph, our negative control in
Section~\ref{sec:controls} --- so, together with the first remark, $16$ is the smallest number of
vertices at which non-Schurian graphs exist, and no enumeration
argument could extend the all-graphs statement there. Non-Schurian graphs
at $n = 16$ are, moreover, not exhausted by the Shrikhande graph: the
companion note \cite{Companion} exhibits seven further ones --- six
graphs of torus triangulations, and one of a different origin: an
$8$-regular graph whose coherent closure is a non-Schurian association
scheme.
None of this reaches the polyhedral class: the
Shrikhande graph is not
planar, having $48 > 3\cdot 16 - 6$ edges, the six triangulation
graphs of \cite{Companion} have $48$ edges as well, and the seventh
has $64$; so for the polyhedral class the
question remains open, and settling it is exactly what the computation does.

\paragraph{Contribution.}
We screen the Schurian property directly on the graph side, for the full
polyhedral class up to $16$ vertices.

\begin{theorem}[computational]\label{thm:main}
Every polyhedral graph with at most $16$ vertices is Schurian.
\end{theorem}

\begin{proof}
By the exhaustive computation described in Section~\ref{sec:method},
covering all \num{413024847068} polyhedral graphs with at most $16$
vertices; no non-Schurian graph occurs.
\end{proof}

\begin{corollary}\label{cor:dim}
Every polyhedral graph with at most $16$ vertices has Weisfeiler--Leman
dimension at most~$2$.
\end{corollary}

\begin{proof}
Immediate from Theorem~\ref{thm:main} and Theorem~\ref{thm:lpz}.
\end{proof}

No planar graph of Weisfeiler--Leman dimension~$3$ is known; by
Corollary~\ref{cor:dim}, a polyhedral example, if one exists, has at
least $17$ vertices.

Two levels must be kept apart when judging what is new here, and we state
the distinction once, bindingly for the rest of the note.

\emph{Enumeration level.} That there are exactly \num{387591510244}
polyhedral graphs on $16$ vertices is not our result; the counts are
published, and we use them as anchors, not as findings. They appear in
Appendix~E of the \texttt{plantri} guide \cite{plantri} and as sequence
A000944 \cite{OEIS}. Reproducing them exactly is precisely what makes them
usable as completeness gates.

\emph{Schurity level.} The per-graph Schurity measurement is, to the best
of our knowledge, the first computation of its kind on this class. For
$n\le 13$ its \emph{outcome} is anticipated by \cite{KlinZivAv2017,ZivAv2018},
so those cases are calibration --- with a residual verification value we
make explicit in Section~\ref{sec:limits}: Ziv-Av states himself that no
formal proof of correctness of his programs is offered, and that external
agreement extends only to order~$13$ \cite[\S2.3]{ZivAv2018}. Our screening
is an independent check by a different route (graph-side screening rather
than configuration-side enumeration), and from $n=14$ on the first one
beyond that externally corroborated range.

\paragraph{Method in three sentences.}
Graphs are generated by \texttt{plantri} in the isomorph-free polyhedral
mode, split into disjoint residue classes for parallel processing. Each
graph first passes a one-sided filter: if the stable $1$-WL vertex colouring
is discrete, the graph is trivially Schurian (Proposition~\ref{prop:filter}),
and no further work is needed; its rate grows with $n$, from $91.7\%$
aggregated over the range $n\le 10$ to $99.99\%$ at $n=16$. Every remaining graph receives
the full verdict, comparing the number of classes of the stable $2$-WL
partition of $\Omega^2$ with the number of $\mathrm{Aut}(X)$-orbits on
$\Omega^2$, both computed exactly.

\section{Preliminaries}\label{sec:prelim}

We use the terminology of Chen and Ponomarenko \cite{ChenPonomarenko}. A
\emph{coherent configuration} on a finite set $\Omega$ is a partition
$\mathcal{S}$ of $\Omega^2$ such that (i) the diagonal is a union of
classes, (ii) $\mathcal{S}$ is closed under transposition, and (iii) for
all $R,S,T\in\mathcal{S}$ the number
$|\{w : (u,w)\in R,\ (w,v)\in S\}|$ is the same for all $(u,v)\in T$. The
classes are the \emph{basis relations}, their number is the
\emph{rank}, and the classes contained in the diagonal define the
\emph{fibres}.

For a graph $X$, the coherent closure $\mathrm{WL}(X)$ is obtained from the
initial partition of $\Omega^2$ into diagonal, edges and non-edges by
iterating the refinement
\[
  \mathrm{sig}(u,v)\;=\;\Bigl(c(u,v),\ \bigl[\!\!\bigl[\,(c(u,w),c(w,v))\ :\
  w\in\Omega\,\bigr]\!\!\bigr]\Bigr),
\]
where $c$ is the current colouring and $[\![\cdot]\!]$ denotes a multiset.
Because the signature contains the previous colour, every round is a
refinement, and the process is stable as soon as the number of classes
stops growing.

The one-dimensional refinement used by our filter is the classical colour
refinement: vertices start with their degree and are recoloured by the pair
(own colour, multiset of neighbour colours) until the number of colour
classes is stable; we call the result the \emph{stable $1$-WL colouring}.
The two are related in the standard way, which we use in
Proposition~\ref{prop:filter}: the fibres of $\mathrm{WL}(X)$ refine the
stable $1$-WL colouring. Indeed, each basis relation is contained in the
product of two fibres \cite[Cor.~2.1.7]{ChenPonomarenko} and refines
the edge/non-edge partition, so the signature of $(v,v)$ determines the
degree of $v$ and the multiset of fibres of its neighbours; the
colouring $v\mapsto c(v,v)$ therefore refines the degree colouring and
is stable under one round of colour refinement. The converse fails ---
a regular graph has a single $1$-WL class but may have several fibres.

For a permutation group $G\le\mathrm{Sym}(\Omega)$ let $\mathrm{Orb}(G)$
denote the partition of $\Omega^2$ into the orbits of $G$ acting by
$g\cdot(u,v)=(g(u),g(v))$, i.e.\ into the orbitals ($2$-orbits) of $G$;
in the notation of \cite[Def.~2.2.3]{ChenPonomarenko},
$\mathrm{Orb}(G)=\mathrm{Orb}(G,\Omega^2)$ is the set of basis
relations of the coherent configuration $\mathrm{Inv}(G)$ associated
with $G$ (there, $\mathrm{Orb}(G)$ without second argument denotes the
orbits on points). We write $\mathrm{Orb}(X):=\mathrm{Orb}(\mathrm{Aut}(X))$
and call it the \emph{orbital configuration} of $\mathrm{Aut}(X)$. As
$2$-WL colours are $\mathrm{Aut}(X)$-invariant, $\mathrm{Orb}(X)$ refines
$\mathrm{WL}(X)$; hence
\[
  X \text{ is Schurian}
  \iff |\mathrm{Orb}(X)| = |\mathrm{WL}(X)| ,
\]
and comparing the two class \emph{numbers} suffices --- a fact the
implementation uses and, as a safeguard, re-verifies for every graph it
fully checks (Section~\ref{sec:method}).

\begin{remark}\label{rem:def}
A coherent configuration $\mathcal X$ on $\Omega$ is \emph{Schurian} if
$\mathcal X=\mathrm{Inv}(\mathrm{Aut}(\mathcal X))$, equivalently if
$\mathcal X=\mathrm{Inv}(K)$ for some $K\le\mathrm{Sym}(\Omega)$
\cite[Def.~2.2.9, Thm.~2.2.10]{ChenPonomarenko}, \cite[\S2.3]{LPZ}.
Since $\mathrm{Aut}(\mathrm{WL}(X))=\mathrm{Aut}(X)$
\cite[Cor.~2.6.6]{ChenPonomarenko}, the graph $X$ is Schurian in the
sense of \cite[p.~2]{LPZ} if and only if $\mathrm{WL}(X)=\mathrm{Orb}(X)$,
which is the form tested here.
\end{remark}

\section{Method and correctness safeguards}\label{sec:method}

\subsection{Generation}

Graphs are produced by \texttt{plantri} \cite{plantri} (version 5.8) in the
mode \texttt{-p -c3}, which generates plane simple graphs of connectivity
at least~$3$ --- exactly the polyhedral graphs --- without isomorphic
duplicates. Isomorph-freeness there is defined with respect to the embedding
on the sphere; by Whitney's theorem \cite{Whitney1932} a $3$-connected
planar graph has an essentially unique embedding, so this coincides with
isomorph-freeness of the abstract graphs. The binary \texttt{planar\_code}
stream is parsed directly; the parser checks adjacency symmetry and vertex
count for every graph and aborts on any violation.

For parallelism we use \texttt{plantri}'s residue splitting
\texttt{res/mod}, which divides the object set into \texttt{mod} disjoint
classes and generates only the \texttt{res}-th, with isomorph-freeness
preserved across classes. The split granularity was $64,\ 256,\ 512,\
1024,\ 4096,\ 16384$ for $n=11,\dots,16$; the value for $n=16$ was chosen
so that each class still contains roughly $150$ of the
$\num{2406841}$ splitting cases that \texttt{plantri}'s own splitting-cost
output reports at internal level~$15$, keeping class sizes reasonably
uniform. (That this is exactly the number of $3$-connected plane
triangulations on $15$ vertices is no accident: the generator reaches the
polyhedral graphs by removing edges from triangulations, and splits one
level below the target size.)

For $n=15$ and $n=16$ the generator was run through a pipe (avoiding
some $39$~TB of transient writes), which under MinGW requires a
one-line binary-mode shim for \texttt{stdout}: the C runtime opens
\texttt{stdout} in text mode and would translate every \texttt{0x0a}
byte of the binary stream into \texttt{0x0d0a}. The generator source
is unchanged; the shim, and the check that pipe and file output agree
byte for byte for $n=10$ and $n=11$, are documented in the README of
the ancillary files. The stages up to $n=14$ wrote to a temporary file
instead and are unaffected.

\subsection{The one-sided filter}

\begin{proposition}\label{prop:filter}
If the stable $1$-WL vertex colouring of $X$ is discrete, i.e.\ assigns $n$
pairwise distinct colours, then $X$ is Schurian, $\mathrm{Aut}(X)=1$, and
both $\mathrm{WL}(X)$ and $\mathrm{Orb}(X)$ consist of $n^2$ singletons.
\end{proposition}

\begin{proof}
(i) Automorphisms preserve stable $1$-WL colours. If all colours are
distinct, every automorphism fixes every vertex, so $\mathrm{Aut}(X)=1$ and
its orbits on $\Omega^2$ are the $n^2$ singletons. (ii) The fibres of
$\mathrm{WL}(X)$ refine the stable $1$-WL classes, by the domination
observed in Section~\ref{sec:prelim}. Under the hypothesis every $1$-WL
class is a singleton, hence so is every fibre. Every basis relation is contained
in the product of two fibres, so every basis relation is a single pair
$\{(u,v)\}$ and $\mathrm{WL}(X)$ consists of $n^2$ singletons as well.
(iii) The two partitions coincide, hence $X$ is Schurian.
\end{proof}

The filter is deliberately \emph{one-sided}: a non-discrete colouring
yields no conclusion and the full check is performed. Its hit rate is
therefore a performance parameter and not a correctness risk. Two empirical
gates guard the implementation nonetheless. First, on the vertex-transitive
control graphs of Section~\ref{sec:controls} the filter must never fire ---
otherwise it could hide the negative control from the full check. Second, in
a validation mode both branches are computed for every graph and compared:
whenever the filter fires, the full check must confirm Schurity with
$|\mathrm{Aut}(X)|=1$ and $|\mathrm{WL}(X)|=|\mathrm{Orb}(X)|=n^2$. This was
run over all \num{440564} graphs on $11$ vertices with zero mismatches, and
over all graphs on $8$ and $9$ vertices in the test suite.

\subsection{The full verdict}

For a graph surviving the filter, three objects are computed exactly.

\emph{Stable $2$-WL partition.} The refinement above is iterated on ordered
pairs including the diagonal until the class count stagnates; each round
costs $O(n^3\log n)$. Only the number of classes enters the verdict,
never the colour identifiers, which are not canonical.

\emph{Automorphism group.} All automorphisms are enumerated explicitly by
backtracking over the $1$-WL colour classes: vertices are assigned in order
of increasing colour-class size, candidate images are restricted to the same
$1$-WL class (legitimate, as $1$-WL colours are $\mathrm{Aut}$-invariant),
and each partial assignment is checked against all previously assigned
vertices. The result is the complete list of permutations, so
$|\mathrm{Aut}(X)|$ is obtained as an exact count rather than as a product
of orbit lengths.

\emph{Orbit partition.} The orbits of $\mathrm{Aut}(X)$ on $\Omega^2$ are
computed by union--find over the $n^2$ cells, uniting $(u,v)$ with
$(g(u),g(v))$ for every enumerated permutation $g$.

The verdict is $|\mathrm{WL}(X)| = |\mathrm{Orb}(X)|$. Two structural
assertions are checked for every graph before the verdict is formed: that
every orbit class lies entirely within one WL class (the refinement
direction, which must hold mathematically), and that
$|\mathrm{Orb}(X)| \ge |\mathrm{WL}(X)|$. A violation of either would
indicate an implementation error; the program then aborts rather than
emitting a verdict.

Enumeration is capped at $10^6$ automorphisms. A cap hit does \emph{not}
produce a verdict: it raises a dedicated error, is logged, and makes the
run report an open verdict with a non-zero exit code. No capped graph
occurred in any run reported here.

\subsection{Control graphs}\label{sec:controls}

Before any screening the checker is exercised on graphs whose verdicts are
fixed mathematically.

\begin{itemize}\itemsep2pt
\item \emph{Shrikhande graph} $\mathrm{Cay}(\mathbb{Z}_4^2,\{\pm(1,0),
  \pm(0,1),\pm(1,1)\})$, an $\mathrm{SRG}(16,6,2,2)$ with
  $|\mathrm{Aut}|=192$; its coherent closure is non-Schurian
  \cite[Example~2.6.17]{ChenPonomarenko} (stated there for the complementary
  $\mathrm{SRG}(16,9,4,6)$). As a strongly regular graph its $2$-WL partition
  has exactly $3$ classes, while the vertex stabiliser has order~$12$ and
  cannot act transitively on the $9$ non-neighbours, so it has at least $4$
  orbits: a \emph{negative} control, the checker must report ``not Schurian''.
\item \emph{$4\times4$ rook's graph} $=H(2,4)$. By Shrikhande's uniqueness
  theorem \cite{Shrikhande1959} these two are the only graphs with parameters
  $(16,6,2,2)$ \cite[\S3]{KlinZivAv2017}; the rook's graph has
  $|\mathrm{Aut}|=1152$ and a rank-$3$ group, giving a \emph{positive}
  control with $3=3$.
\item \emph{Dodecahedral graph}, polyhedral, distance-transitive,
  $|\mathrm{Aut}|=120$, diameter $5$: a positive control inside the target
  class, with $6$ orbits on $\Omega^2$. It is also the one Platonic graph whose
  Schurity Li, Ponomarenko and Zeman note can be settled by a straightforward
  computation \cite[p.~2]{LPZ}.
\end{itemize}

The Shrikhande/rook pair also controls the discriminating power of the
tooling: $2$-WL cannot separate the two graphs, yet they are not isomorphic.
All these verdicts, including the group orders $192$, $1152$ and $120$, are
fixed as automated tests.

\subsection{Run-level gates}\label{sec:gates}

\begin{enumerate}\itemsep3pt
\item \emph{Double count per split.} \texttt{plantri} reports on
  \texttt{stderr} how many objects it wrote; this number must equal our
  parser's count for every split, otherwise the run stops. A missing or
  unparsable report is itself an error.
\item \emph{Completeness anchor per $n$.} The sum over all splits must equal
  the published count of polyhedral graphs on $n$ vertices exactly. Every
  $n$ reported here passed this gate. Every checkpoint file was in addition
  re-read by a separate script that shares no code with the pipeline and
  re-derives all published figures from the raw records; for $n=16$ it
  confirms $16384$ of $16384$ splits present, no duplicates, no parse errors,
  sum exact. That script is included with the ancillary files.
\item \emph{Seeded re-enumeration.} For $n=15$ and $n=16$ a $1\%$ sample of
  splits, drawn with seed \texttt{20260724}, was recomputed in a separate run
  and compared per split: $40/40$ and $163/163$ splits agreed in all five
  counting fields (count, filtered, non-Schurian, capped, mismatches),
  covering \num{234435128} and \num{3777936313} graphs respectively. This is
  a tuple comparison, not a bitwise comparison of outputs. Its purpose is to
  guard against silent bit errors: the machine has no ECC memory.
\item \emph{Independent group orders.} Graphs were passed to \texttt{nauty}
  \cite{nauty} and its group order compared with ours. Two samplings were
  used: an exhaustive one over all \num{2606} graphs on $9$ vertices, and a
  strided one (seed-chosen splits at granularity $4096$, every $k$-th graph
  within them), giving \num{25162} graphs for $n=13$, \num{10657} for
  $n=14$, \num{10031} for $n=15$ and \num{10203} for $n=16$. All comparisons
  agreed; see Section~\ref{sec:limits} for what
  the strided samples do and do not reach. The tool is first anchored on the
  control graphs, where it must return $192$, $1152$ and $120$.
\item \emph{Interpreter gate.} The stages up to $n=14$ ran under CPython,
  those for $n=15$ and $n=16$ under PyPy for speed. Two full stages, and
  one split of a third, were cross-checked against the other interpreter
  with per-split identical counts: the full runs for $n=12$ and $n=13$ were repeated under PyPy, and
  split~$0$ of $n=16$ --- first computed under CPython in a timing run ---
  was recomputed under PyPy at $2.2$ times the speed on identical counts.
  For $n=14$ and $n=15$ no such repetition was made. Splits $0$ and $1$ of
  the $n=16$ checkpoint stem from the CPython timing run and were retained
  there rather than overwritten by the repetition; the remaining
  \num{16382} splits ran under PyPy. A one-seed subsample of \num{5336} of the
  $n=13$ \texttt{nauty} graphs was drawn identically under both
  interpreters and agreed.
\item \emph{Crash safety.} Checkpoints are rewritten atomically after every
  completed split and validated strictly on resume: corruption, duplicate
  splits or mismatched parameters abort the run instead of being skipped.
  Findings are flushed and \texttt{fsync}ed immediately.
\item \emph{Filter-negative cross-check.} After the production runs, a
  seeded sample of splits ($40$ of $16384$ at $n=16$ and $20$ of $4096$
  at $n=15$, seed \texttt{20260828}) was re-generated and re-filtered;
  the per-split count and filter fields matched the production
  checkpoints exactly for all $60$ splits, re-covering
  \num{1050294199} graphs. Every one of the \num{83189}
  filter-negative graphs in these splits --- the population in which
  all symmetric graphs live; \num{83176} of them have
  $|\mathrm{Aut}(X)| > 1$, with group orders up to $24$ --- was then
  checked on three axes: the full production verdict was recomputed;
  the enumerated automorphism list was certified against
  \texttt{nauty} --- its length against the \texttt{nauty} group
  order for every single graph, the list verified duplicate-free, and
  each listed map verified
  individually to be a bijection preserving adjacency, so the list is
  pinned element
  by element and, through the order, in its entirety; and the $2$-WL
  class count was re-derived by an independent reimplementation (the
  one used in the verification battery of the companion note
  \cite{Companion}), with the orbit count re-counted by a second
  implementation over the certified list. No deviation
  occurred on any axis, and every graph was confirmed Schurian.
\end{enumerate}

\section{Results}\label{sec:results}

Table~\ref{tab:results} gives the screening results. The filter column
reports how many graphs were disposed of by Proposition~\ref{prop:filter};
the remaining ``full checks'' each received the complete verdict of
Section~\ref{sec:method}. No graph was non-Schurian, and no verdict was left
open by a cap, at any $n$.

\begin{table}[ht]\centering
\begin{tabular}{rrrrrr}
\toprule
$n$ & polyhedral graphs & $1$-WL discrete & full checks & filter rate &
non-Schurian\\
\midrule
$4\text{--}10$ & \num{35207}          & \num{32269}          & \num{2938}      & $91.66\%$ & $0$\\
$11$           & \num{440564}         & \num{430494}         & \num{10070}     & $97.71\%$ & $0$\\
$12$           & \num{6384634}        & \num{6335968}        & \num{48666}     & $99.24\%$ & $0$\\
$13$           & \num{96262938}       & \num{96039167}       & \num{223771}    & $99.77\%$ & $0$\\
$14$           & \num{1496225352}     & \num{1495112910}     & \num{1112442}   & $99.93\%$ & $0$\\
$15$           & \num{23833988129}    & \num{23828671708}    & \num{5316421}   & $99.98\%$ & $0$\\
$16$           & \num{387591510244}   & \num{387564497271}   & \num{27012973}  & $99.99\%$ & $0$\\
\midrule
total          & \num{413024847068}   & \num{412991119787}   & \num{33727281}  &           & $0$\\
\bottomrule
\end{tabular}
\caption{Screening of all polyhedral graphs with $4\le n\le 16$. Counts of
polyhedral graphs are the published anchors \cite{plantri,OEIS}, matched
exactly by every run. No polyhedral graph with $n\le 6$ has a discrete
stable $1$-WL colouring, so the filter first fires at $n=7$; from there its
rate grows monotonically with $n$.}
\label{tab:results}
\end{table}

The two largest stages dominate the cost. Summing the per-split wall-clock
times recorded in the checkpoints gives $\num{180}$ worker hours for $n=15$
and $\num{3110}$ for $n=16$, i.e.\ about $6.9$ days on $20$ worker
processes of a $24$-thread consumer machine, spread over the period from
July~24 to August~3, 2026 (the sum includes idle inflation from
suspends and throttling). The run was interrupted several times,
including one hard power loss; the seeded re-enumeration was performed
\emph{after} all of these and reproduced the sampled splits field by
field, which is evidence against silent corruption in the checkpointed
results.

\section{Scope and limitations}\label{sec:limits}

We state the boundaries of the result explicitly.

\begin{enumerate}\itemsep3pt
\item \emph{Empirical, bounded.} Theorem~\ref{thm:main} is established by
  exhaustive computation rather than by a mathematical argument; with respect
  to the belief expressed in \cite[p.~2]{LPZ} it is instance evidence only.
  Every formulation must carry the bound $n\le 16$.
\item \emph{Class.} Only $3$-connected planar graphs were screened. The
  question for planar graphs without the connectivity assumption is
  untouched here.
\item \emph{Where the novelty lies.} For $n\le 13$ the outcome also follows
  from \cite{KlinZivAv2017,ZivAv2018}; our contribution there is an
  independent route, which has residual value because that enumeration is
  itself computational and, by its author's own statement, offered without a
  formal correctness proof and externally corroborated only up to order~$13$
  \cite[\S2.3]{ZivAv2018}. From $n=14$ on, a derivation from the
  configuration side would need the additional step discussed in
  Section~\ref{sec:intro}, which is unpublished as far as we can tell; at
  $n=16$ no such derivation is possible for graphs in general, since the
  Shrikhande graph realises a non-Schurian closure, and the polyhedral
  restriction carries the whole statement. The counts of
  polyhedral graphs themselves are reproduction, not a finding.
\item \emph{Weisfeiler--Leman dimension.} Corollary~\ref{cor:dim} is a
  consequence of the published Theorem~\ref{thm:lpz}; we did not measure
  $\dim_{\mathrm{WL}}$ ourselves.
\item \emph{Re-enumeration coverage.} The seeded $1\%$ re-enumeration was
  carried out for $n=15$ and $n=16$ only. For $n\le 14$ --- \num{1599348695}
  of the \num{413024847068} graphs --- it was never run; the infrastructure
  exists and the gap is closable. Those stages are not unguarded, though: for
  $n=12$ and $n=13$ the repeat runs of Section~\ref{sec:gates}, and for
  $n=11$ the validation run of Section~\ref{sec:method}, cover every graph,
  and all stages carry the anchor counts and the double count per split.
\item \emph{A field that carries no weight.} The checkpoints contain a
  \texttt{mismatches} field which is structurally zero on the production
  path, since it is incremented only in the validation mode described in
  Section~\ref{sec:method}, which the production and re-enumeration runs do
  not use. It must not be read as an independent verification signal; the
  second line of defence consists of the anchor gates, the re-enumeration and
  the \texttt{nauty} cross-check.
\item \emph{Reach of the cross-check.} Two caveats apply to the
  \texttt{nauty} comparison, and the second is the more serious one. It was
  drawn for $n=16$ from the same split list as for $n=15$, only with a
  different stride, so it is correlated with the split structure rather than
  independently sampled. More importantly, it samples uniformly from all
  graphs, and since at least $99.97\%$ of them are filtered out as
  $1$-WL-discrete and hence have a trivial automorphism group, the expected
  number of sampled graphs with $|\mathrm{Aut}(X)|>1$ is only about $2$ for
  $n=15$ and below $1$ for $n=16$. What the sample verifies at those sizes is
  therefore essentially the trivial-group case. This gap has since
  been closed by the filter-negative cross-check of
  Section~\ref{sec:gates}: in the sampled splits, every
  filter-negative graph --- the population in which all symmetric
  graphs live --- had its automorphism list certified against
  \texttt{nauty}, its $2$-WL class count re-derived by an independent
  implementation, and its verdict recomputed, with no deviation among
  \num{83189} graphs. The remaining evidence stands as before: all
  \num{2606} graphs on $9$
  vertices were compared with \texttt{nauty} exhaustively, the control graphs
  pin the group orders $192$, $1152$ and $120$, and the validation runs cover
  every graph on $8$, $9$ and $11$ vertices.
\item \emph{Trusted components.} The generator, the isomorph-freeness of its
  output, and \texttt{nauty}'s group orders are taken as given. They are
  anchored --- by the published counts, by the double count, and by the
  control graphs --- but not re-derived. The anchoring is weaker than it
  looks for the two largest stages: the published counts for $n=15$ and
  $n=16$ originate from the same generator (they are stated in its guide),
  so agreement there confirms our pipeline rather than the generator.
  Independent enumerations predating that generator exist, but they reach far
  smaller sizes: Duijvestijn and Federico \cite{DuijvestijnFederico} tabulate
  polyhedral graphs by edge count up to $22$ edges and, with additional data
  beyond that, give the exact count for $10$ faces --- equivalently, by
  duality, for $10$ vertices --- while for larger sizes they give estimates
  only: their projection for $12$ faces was ``in the neighborhood of six and
  a half million'', against the exact \num{6384634}. From $n=11$ onwards our anchors therefore rest on
  tables published together with the generator, and the agreement confirms
  our pipeline rather than those tables.
\item \emph{One implementation of the closure.} The cross-check with
  \texttt{nauty} compares automorphism group orders only; the coherent
  closure itself rests on a single implementation on the production
  path, exercised on the control
  graphs. A systematic error making the closure too coarse would produce
  false alarms rather than false negatives, and an error equating closure
  and orbit partition is caught by the Shrikhande control, which must come
  out non-Schurian. On the sampled filter-negative population, the
  class counts have since been re-derived by an independent second
  implementation of $2$-WL with no deviation (the filter-negative
  cross-check of Section~\ref{sec:gates}); outside that sample the
  single-implementation caveat stands. The same applies, with
  more weight, to the filter: it settles $99.99\%$ of all verdicts, and it
  was held against the full check only for $n=8$, $9$ and $11$, that is on
  \num{443427} graphs. A filter that wrongly fires would produce false
  Schurity verdicts, and beyond $n=11$ nothing but the control graphs of
  Section~\ref{sec:controls} stands against that.
\end{enumerate}

\section{Reproducibility}\label{sec:repro}

The computation was carried out on a $24$-thread consumer machine under
Windows~11 with $64$~GB of memory without ECC, using CPython~3.12.10 for
$n\le 14$ and PyPy~3.11 (v7.3.23) for $n=15$ and $n=16$, with the two
exceptional splits recorded in Section~\ref{sec:gates}; the generators were
built with MinGW-w64 gcc~16.1.

The screening modules are written in Python and use the standard library
only. They accompany the arXiv version of this note as ancillary files ---
and are available on request before that --- together with the per-split
checkpoint files, the control-graph tests, the verification script mentioned
in Section~\ref{sec:gates}, and a README giving the exact command lines and
split granularities. Three files there differ minimally from the versions used
in the runs --- two to make the package self-contained without the author's
wider codebase, one so that the generator-dependent tests are skipped rather
than failing when the binary is absent --- and all three are listed in its
manifest. The external tools run
as subprocesses without a shell, namely \texttt{plantri} 5.8 and
\texttt{nauty} 2.8.9, with the source tarballs identified by

{\small\ttfamily
\begin{tabbing}
\hspace{2.2cm}\=\kill
plantri58 \> e78a944116fec9f2c9f5e484206276cc2b\\
          \> 0043bae803e9815f4b2683614629b8\\[2pt]
nauty2\_8\_9 \> c97ab42bf48796a86a598bce3e9269047c\\
          \> a2b32c14fc23e07208a244fe52c4ee
\end{tabbing}}

\noindent (SHA-256, split across two lines each). Neither tool's source was
altered. Two small compatibility shims of our own are linked into the
generator build under MinGW and are included with the ancillary files: a
replacement for \texttt{sys/times.h}, without which the generator does not
compile there, and the constructor setting \texttt{stdout} to binary mode
described in Section~\ref{sec:method}. Given the same generator
version, any split can be recomputed independently: the residue splitting is
deterministic, so a reader who wants to check a single split of the largest
stage can do so in minutes, without relying on our checkpoints at all.

\section*{Acknowledgments}

This work was carried out in close collaboration with the AI assistant
Claude (Anthropic; models Claude Fable~5 and Claude Opus~5, 2026), which
contributed to the implementation, the verification design, the
verification of the bibliography against the cited sources, and the
drafting of this note. Every
computational claim reported here is gated by the safeguards described in
Section~\ref{sec:method}; the numbers were re-derived from the raw
checkpoint files in a separate pass, and the manuscript underwent repeated
adversarial reviews by fresh model instances blind to the development
history. The author bears sole responsibility for the content.

\end{document}